\documentclass [a4paper,12pt]{article}
\usepackage{amsmath}
\usepackage{amsfonts}
\usepackage{indentfirst}
\usepackage{esint}
\usepackage{lipsum}
\usepackage{latexsym}
\usepackage{bbm}
\usepackage{amsfonts}
\usepackage{mathrsfs}
\usepackage{amssymb}
\usepackage{amsmath}
\usepackage{amsthm}
\usepackage{latexsym}
\usepackage{indentfirst}
\usepackage{enumitem}
\usepackage{bm}
\usepackage{graphicx}
\usepackage{esint}
\usepackage{hyperref}
\usepackage{cleveref}
\usepackage{cite}
\usepackage[numbers,square]{natbib}
\numberwithin{equation}{section}
\usepackage{color}
\allowdisplaybreaks
\usepackage{enumerate}
\usepackage{ulem}
\allowdisplaybreaks

\newtheorem{theorem}{Theorem}[section]
\newtheorem{lemma}[theorem]{Lemma}
\newtheorem{thm}[theorem]{Theorem}
\newtheorem{prop}[theorem]{Proposition}
\newtheorem{cor}[theorem]{Corollary}
\newtheorem{defn}[theorem]{Definition}
\newtheorem{rmk}[theorem]{Remark}

\makeatletter
\usepackage{amsmath}
\usepackage{amsfonts}
\usepackage{indentfirst}
\usepackage{esint}
\usepackage{latexsym}
\usepackage{authblk}
\usepackage{color}
\UseRawInputEncoding
\makeatletter

\newcommand{\Rmnum}[1]{\expandafter\@slowromancap\romannumeral #1@}
\newcommand\blfootnote[1]{%
  \begingroup
  \renewcommand\thefootnote{}\footnote{#1}%
  \addtocounter{footnote}{-1}%
  \endgroup
}
\makeatother

\begin{document}
\title{Convergence Rates for the Stationary and Non-stationary Navier-Stokes Equations over Non-Lipschitz Boundaries}
\author{Yiping Zhang\thanks{School of Mathematics and Statistics, Central China Normal University, 430079 Wuhan, Hubei,
People’s Republic of China. Email: zhangyiping161@mails.ucas.ac.cn}}
\date{}
\maketitle
\begin{abstract}
In this paper, we consider the higher-order convergence rates for the 2D stationary and non-stationary Navier-Stokes
Equations over highly oscillating periodic bumpy John domains with $C^{2}$ regularity in some neighborhood of the boundary point (0,0). For the stationary case and any $\gamma\in (0,1/2)$,  using the variational equation satisfied by the solution and the correctors for the bumpy John domains obtained by
Higaki,  Prange and Zhuge \cite{higaki2021large,MR4619004} after correcting the values on the  inflow/outflow boundaries
$(\{0\}\cup\{1\})\times(0,1)$, we can obtain an $O(\varepsilon^{2-\gamma})$ approximation in $L^2$ for the
velocity and an $O(\varepsilon^{2-\gamma})$ convergence rates in $L^1$ approximated by the so called Navier's wall
laws, which generalized the results obtained by  J\"{a}ger and Mikeli\'{c} \cite{MR1813101}. Moreover, for the non-stationary case, using the energy method, we can obtain an $O(\varepsilon^{2-\gamma}+\exp(-Ct))$ convergence rate for the velocity in $L_x^1$.

\end{abstract}
\section{Introduction}
\noindent\blfootnote{ Mathematics Subject Classification: 35Q30, 35B27, 76D10, 76M50\\
Keywords: stationary and non-stationary Navier-Stokes equations, John domains, higher-order convergence rates.}

Over the past few decades, the effective boundary condition and the convergence rates for the flows over
bumpy domains have been extensively investigated. The effective boundary condition, also called the Navier's wall laws, is usually used for simulation of flows over bumpy domains. As early as in 1827, Navier \cite{navier1827lois} claimed that
the slip velocity should be proportional to the shear stress. The idea is to replace the no-slip
condition at rough boundaries with the non-penetration condition plus a relation between the tangential
velocity and the shear stress, which can be rigorously derived in \cite{MR1657773} for the Laplace operator in an annular
domain with rough perforations. See also the papers \cite{MR1322334,MR1648212,MR1657773,MR2286014,MR1602088} for an extensive reference for the wall laws.

Over the last two decades, the theories for the Stokes flows and the stationary Navier-Stokes flows over
bumpy domains have been extensively understood. In \cite{MR1813101}, J\"{a}ger and  Mikeli\'{c} considered
the 2D laminar viscous channel flow with the lateral surface of the channel containing surface
irregularities. They rigorously obtained the Navier friction condition and the $O(\varepsilon^{3/2-})$ (the definition of $O(\varepsilon^{3/2-})$ will be given later)
convergence rates with the help of the effective solutions. See also the 3D case \cite{MR1952473} with
the rough domain given by a Lipschitz graph, where an $O(\varepsilon^2)$ convergence rates in $L^1$ was
obtained (the similar results as in Corollary \ref{c1.2}), and the stochastic case \cite{MR2410410,MR2470924}.

Moreover, there are many works focusing on the regularity of flows or equations over bumpy domains. For the linear elliptic systems or equations in
divergence form with periodically oscillating coefficients over domains satisfying the so-called
$\varepsilon$-scale flatness, which could be arbitrarily rough below $\varepsilon$-scale, Zhuge
\cite{MR4199276} obtained a large-scale Lipschitz regularity estimate, see also the previous works by Kenig and Prange \cite{MR3325774,MR3784804}. For the Stokes equation over bumpy John domains, Higaki and Zhuge
\cite{MR4619004} investigated  the large-scale boundary regularity for the Stokes system in periodically
oscillating John domains and constructed boundary layer correctors of arbitrary order, which implies the
large-scale regularity estimate, as well as a Liouville theorem, of arbitrary
order for the Stokes system and the higher-order boundary
layer tails and wall laws in viscous fluids over rough boundaries. For the stationary Navier-Stokes Equations, Higaki and Prange \cite{MR4123336} obtained the large-scale $C^{1,\beta}$ regularity with the rough domain given by a Lipschitz graph. Later on, Higaki-Prange-Zhuge \cite{higaki2021large} considered this problem over bumpy John domains and obtained the large-scale $C^{2,\beta}$ regularity.

Besides the works above, there are many other works focusing on the flows over rough domains, such as the wall laws for the unsteady incompressible Navier-Stokes equations \cite{MR1906814} and for the compressible flows \cite{MR3642228,MR3335531}.
Moreover, we refer the readers to references in this paper and the  references  therein for more results.

\subsection{Main Results}
Now, we introduce our main results and we start by fixing the problem setting.

We consider the laminar viscous two-dimensional stationary and non-stationary Navier-Stokes Equations through a domain $\Omega_\varepsilon$ consisting of the channel $\Omega_0=:(0,1)\times(0,1)$, the interface $\Gamma_0=:(0,1)\times\{0\}$, and the layers of roughness $\Omega_\varepsilon\setminus\Omega_0$. The exact definition of the rough domain $\Omega_\varepsilon$ will be given in next subsection.

Denote $\Gamma_0=:(0,1)\times \{0\}$, $\Gamma_1=:(0,1)\times \{1\}$, $\Gamma_\varepsilon=:\partial \Omega_\varepsilon\setminus\partial\Omega_0$, $\Sigma_0=:\{0\}\times (0,1)$ and $\Sigma_1=:\{1\}\times (0,1)$.
For $0<\varepsilon \ll 1$, we consider the following stationary Navier-Stokes Equations:
\begin{equation}\label{1.1}
\left\{\begin{aligned}
-\Delta U_{S,\varepsilon}+U_{S,\varepsilon}\cdot \nabla  U_{S,\varepsilon}+\nabla P_{S,\varepsilon}&=0 \quad \text{in }\Omega_\varepsilon,\\
\operatorname{div}U_{S,\varepsilon}&=0 \quad \text{in }\Omega_\varepsilon,\\
U_{S,\varepsilon}&=0 \quad \text{on }\Gamma_\varepsilon\cup \Gamma_1,\\
\ U_{S,\varepsilon,2}&=0 \quad \text{on }\Sigma_0\cup \Sigma_1,\\
P_{S,\varepsilon}=p_0 \text{ on }\Sigma_0&\quad \text{ and }\quad\ P_{S,\varepsilon}=p_1 \text{ on }\Sigma_1,
\end{aligned}\right.\end{equation}
for some constants $p_0$ and $p_1$.

Now our aim is to investigate the effective behavior of the velocities $U_{S,\varepsilon}$ as $\varepsilon\rightarrow 0$, which means that the characteristic size of the irregularities tends to zero.

It is obvious that in $\Omega_0$ the flow continues to be governed by the stationary Navier-Stokes Equations. Moreover, the presence of the irregularities would only contribute to the effective boundary conditions at the lateral boundary \cite{MR1813101,MR1952473,MR4123336,higaki2021large}, and the main goal of this paper is to obtain the higher order convergence rates as $\varepsilon\rightarrow0$, compared to the case considered in \cite{MR1813101,MR1952473}.

Recall that the classic Poiseuille flow in $\Omega_0$ is given by
\begin{equation}\label{1.2}
(U_{S,0},P_{S,0})=\left(\frac12(p_1-p_0)x_2(x_2-1),0,(p_1-p_0)x_1+p_0\right),\end{equation}
and we extend $U_{S,0}$ to be 0 for $x\in \Omega_\varepsilon\setminus\Omega_0$.

Moreover, the effective model \cite[Proposition 4]{MR1813101} is defined via the first-order corrector and given by
\begin{equation}\label{1.3}
(U_{S,\text{eff}},P_{S,\text{eff}})=\left(\frac{p_1-p_0}{2}\left(x_2^2-(x_2+\varepsilon\alpha_1)
\frac{1}{1+\varepsilon\alpha_1}\right),0,(p_1-p_0)x_1+p_0)\right),\end{equation}
with the constant $\alpha=(\alpha_1,0)$ defined in Proposition \ref{p2.3}.

Throughout this paper, we always denote $h^\varepsilon(x)=:h(x/\varepsilon)$ if the content is understood and the notation $C\varepsilon^{2-}$ means $C_\delta\varepsilon^{2-\delta}$ for any $1/2>\delta>0$, and $C_\delta\rightarrow\infty$ as $\delta\rightarrow 0$.

Now, the first goal is to obtain higher order convergence rates with the domain $\Omega_\varepsilon$ being more irregular than the case considered in \cite{MR1813101,MR1952473}. Moreover, with the help of $U_{S,\text{eff}}$, we can obtain the $O(\varepsilon^{3/2})$ convergence rates in $L^2$ and $O(\varepsilon^{2-})$ convergence rates in $L^1$, stated in Corollary \ref{c1.2}.
\begin{thm}\label{t1.1}
Let the domain $\Omega_\varepsilon$ be defined in Definition \ref{d1.9}, then
there exists a universal small constant $C_1'<1$, such that if $|p_1-p_0|\leq C_1'$, then
\begin{equation*}
\left\|U_{S,\varepsilon}-U_{S,0}+\frac\varepsilon 2(p_1-p_0)\left(V(x/\varepsilon)-\alpha\right)
+\frac\varepsilon 2(p_1-p_0)\alpha_1(1-x_2)e_1\right\|_{L^2(\Omega_0)}\leq C \varepsilon^{2-},
\end{equation*}
for the constant $C$ independent of $\varepsilon$, where the function $V$ is the first order corrector defined in \eqref{2.6}.
\end{thm}

As a direct corollary of Theorem \ref{t1.1}, we have the following convergence rates:
\begin{cor}\label{c1.2}
Under the conditions in Theorem \ref{t1.1}, there holds
\begin{equation*}
\left\|U_{S,\varepsilon}-U_{S,\text{eff}}+\frac\varepsilon 2(p_1-p_0)\left(V(x/\varepsilon)-\alpha\right)
\right\|_{L^2(\Omega_0)}\leq C \varepsilon^{2-},
\end{equation*}
\begin{equation*}
\left\|U_{S,\varepsilon}-U_{S,\text{eff}}
\right\|_{L^2(\Omega_0)}\leq C \varepsilon^{3/2},
\end{equation*}
\begin{equation*}
\left\|U_{S,\varepsilon}-U_{S,\text{eff}}
\right\|_{L^1(\Omega_0)}\leq C \varepsilon^{2-},
\end{equation*}
for the constant $C$ independent of $\varepsilon$.
\end{cor}

\begin{rmk}\label{r1.3}
(1) Compared with \cite{MR1813101}, the main difference of our paper is stated as following:

$\text{(i)}$: We work in bumpy John domains with $C^{2}$ regularity in some neighborhood of the boundary point $(0,0)$, as defined in Definition \ref{d1.9} that are not necessarily graphs, while in \cite{MR1813101}, the bumpy boundary is given by a Lipschitz continuous boundary with the similar $C^{2}$ regularity.

$\text{(ii)}$: We obtain an $O(\varepsilon^{2-})$ approximation in $L^2$ for the
velocity and an $O(\varepsilon^{2-})$ convergence rates in $L^1$ as well as an $O(\varepsilon^{3/2})$ convergence rates in $L^2$ approximated by the Navier's wall
laws, while in \cite{MR1813101}, only an $O(\varepsilon^{3/2})$ approximation in $L^2$ for the
velocity and an $O(\varepsilon^{3/2-})$  convergence rates in $L^2$ were obtained.\\

(2) J\"{a}ger and Mikeli\'{c} \cite{MR1952473} also considered couette flows (the 3D case) over a rough boundary given by a Lipschitz continuous graph. Similar to the results in Theorem \ref{t1.1} and Corollary \ref{c1.2}, they obtained an $O(\varepsilon^{2})$ approximation for the effective mass
flow and an $O(\varepsilon^{2})$ convergence rates in $L^1$ for the velocity. The solution considered in \cite{MR1952473} is periodic in $(x_1,x_2)$, so there is no need to correct the value of  approximation solution  on the inflow/outflow boundaries
$(\{0\}\cup\{1\})\times(0,1)\times(0,1)$. Thus, there is no  additional regularity assumption on the neighborhood of the boundary point.

\end{rmk}

The second goal of this paper is to investigate the following non-stationary Navier-Stokes Equations:
\begin{equation}\label{1.4}
\left\{\begin{aligned}
\partial_t U_{N,\varepsilon}-\Delta U_{N,\varepsilon}+U_{N,\varepsilon}\cdot \nabla  U_{N,\varepsilon}+\nabla P_{N,\varepsilon}&=0 \quad\text{in }\Omega_\varepsilon\times (0,\infty),\\
\operatorname{div}U_{N,\varepsilon}&=0 \quad \text{in }\Omega_\varepsilon\times (0,\infty),\\
U_{N,\varepsilon}&=0 \quad \text{on }\{\Gamma_\varepsilon\cup \Gamma_1\}\times (0,\infty),\\
\ U_{N,\varepsilon,2}&=0 \quad \text{on }\{\Sigma_0\cup \Sigma_1\}\times (0,\infty),\\
P_{N,\varepsilon}=p_0 \text{ on }\{\Sigma_0\}\times (0,\infty),&\quad\ P_{N,\varepsilon}=p_1 \text{ on }\{\Sigma_1\}\times (0,\infty),\\
U_{N,\varepsilon}(x,0)&=\varphi(x) \text{ on }\Omega_\varepsilon,
\end{aligned}\right.\end{equation}
where $0<\varepsilon \ll 1$, $p_0$ and $p_1$ are given constants, and $\varphi(x)\in L^2(\Omega_\varepsilon)$ is the initial data. For the non-stationary Navier-Stokes Equations \eqref{1.4}, we have:

\begin{cor}\label{c1.4}
Under the conditions in Theorem \ref{t1.1}, we additionally assume that there exists a constant $\delta\in (0,1)$ such that $||\varphi-U_{S,0}||_{L^2(\Omega_\varepsilon)}\leq (1-\delta)/G_N^2$ with $G_N$ a universal constant defined in \eqref{1.5}, then for any $0< \varepsilon\leq c_1\delta^2$ and $|p_1-p_0|\leq c_1\delta$ with $c_1<1/10$ being a universal small constant, we have
\begin{equation*}
\int_{\Omega_\varepsilon}|U_{N,\varepsilon}(\cdot,t)-U_{S,0}(\cdot)|^2\leq C\varepsilon^{2}+2\exp\{-C\delta t\},
\end{equation*}
\begin{equation*}
\int_{\Omega_0}|U_{N,\varepsilon}(\cdot,t)-U_{S,\text{eff}}(\cdot)|^2\leq C\varepsilon^{3}+2\exp\{-C\delta t\},
\end{equation*}
and
\begin{equation*}
\int_{\Omega_0}|U_{N,\varepsilon}(\cdot,t)-U_{S,\text{eff}}(\cdot)|\leq C\varepsilon^{4-}+2\exp\{-C\delta t\},
\end{equation*}
for the constant $C$ independent of $\varepsilon$. Moreover, we also have
\begin{equation*}
\int_{\Omega_0}\left|U_{N,\varepsilon}-U_{S,0}+\frac\varepsilon 2(p_1-p_0)\left(V(x/\varepsilon)-\alpha+\alpha_1(1-x_2)e_1\right)\right|^2(\cdot,t)
\leq C\varepsilon^{4-}+2\exp\{-C\delta t\},
\end{equation*}
for the constant $C$ independent of $\varepsilon$.
\end{cor}

\begin{rmk}\label{r1.5}
(1) Note that $\varphi$ can depend on $x_2$ and $\varphi_2$ does not need to be identical to $0$.

(2) Another interesting problem for \eqref{1.4} is that the rough domain $\Omega_\varepsilon$ changes along with the time variable $t$, which is left for further.
\end{rmk}

At the end of the subsection, we denote $G_N\geq 1$ the universal constant such that the following Galiardo-Nirenberg's inequality holds true:
\begin{equation}\label{1.5}
||u||_{L^4(\mathbb{A}_0)}\leq G_N||u||_{L^2(\mathbb{A}_0)}^{1/2}||\nabla u||_{L^2(\mathbb{A}_0)}^{1/2},
\end{equation}
for any $u=(u_1,u_2)\in H^1(\mathbb{A}_0)$ and $u=0$ on $(0,1)\times \{-1/10\}$ with $\mathbb{A}_0=:(0,1)\times(-1/10,1)$.

\subsection{Notations and Definitions}
In this subsection, we give the notations and definitions used in this paper. We first define John domains. These domains were introduced by John in \cite{MR138225} and named after John in \cite{MR565886}.

\begin{defn}\label{d1.6}
Let $\tilde{\Omega}\subset \mathbb{R}^d$ be an open bounded set and $\tilde{x}\in \tilde{\Omega}$. We say that $\tilde{\Omega}$ is a John domain (or a bounded John domain) with respect to $\tilde{x}$ and with constant $L$, if for any $y\in \tilde{\Omega}$, there exists a Lipschitz mapping $\rho:[0,|y-\tilde{x}|]\rightarrow \tilde{\Omega}$ with Lipschitz constant $L\in (0,\infty)$, such that $\rho(0)=y$, $\rho(|y-\tilde{x}|)=\tilde{x}$ and $\text{dist}(\rho(t),\partial \tilde{\Omega})\geq t/L$ for all $t\in [0,|y-|\tilde{x}|]$.
\end{defn}
The definitions above of bounded domains can be generalized to a class of unbounded domains \cite{higaki2021large}.

\begin{defn}\label{d1.7}
Let $\tilde{\Omega}\subset \mathbb{R}^d$ containing the upper half-space of $\mathbb{R}^d$ and assume $\partial \tilde{\Omega}\subset \{-1\leq x_d\leq 0\}$. We say that $\tilde{\Omega}$ is a bumpy John domain (or a bumpy John half-space) with constants $(L,K)$, if for any $x\in \{x_d=0\}$ and any $R\geq 1$, there exists a bounded John domain $\tilde{\Omega}_R(x)$ with respect to $x_R=x+Re_d$ and with constant $L\in(0,\infty)$ according to Definition \ref{d1.6} such that
\begin{equation*}
B_{R,+}(x)\subset\tilde{\Omega}_R(x)\subset B_{KR,+}(x),
\end{equation*}
where $B_{R,+}(x)=Q_{R}(x)\cap \tilde{\Omega}$. Here $Q_R(x)$ is a cube in $\mathbb{R}^d$ centered at $x$ with side length $2R$.
\end{defn}

\begin{defn}\label{d1.8}
We say that $\tilde{\Omega}$ is a periodic bumpy John domain if the following holds:

(i) $\tilde{\Omega}$ is a bumpy John domain with constant $(L,K)$,

(ii) $\tilde{\Omega}$ is a $\mathbb{Z}$-translation invariant, namely $z+\tilde{\Omega}=\tilde{\Omega}$ for any $z\in\mathbb{Z}^{d-1}\times\{0\}$.
\end{defn}

Now we are ready to introduce the rough domain $\Omega_\varepsilon$ considered in this paper. Note that in Definition \ref{d1.9} we set $d=2$.
\begin{defn}\label{d1.9}
The domain $\Omega$ considered in this paper satisfies the following assumptions:

(i) ${\Omega}$ is a periodic bumpy John domain with constant $(L,K)$.

(ii) We assume the point $(0,0)\in \partial \Omega$ and $\partial \Omega\in C^{2}$ in some neighborhood of the boundary point (0,0).\\
Then the rough domain $\Omega_\varepsilon$ is always defined as
$\Omega_\varepsilon=:\{\varepsilon\Omega\}\cap \{(0,1)\times(-1,1)\}$ with
$\varepsilon\Omega=:\{x\in\mathbb{R}^2|\varepsilon^{-1}x\in \Omega\}$ and $\varepsilon^{-1}\in\mathbb{N}_+$. Note that the layers of roughness
$\left\{\Omega_\varepsilon\setminus\Omega_0\right\}\subset \{[0,1]\times [-\varepsilon,0]\}$.
\end{defn}

At the end of this section, we give an explanation of using Sobolev-Poinc\'{a}re's inequality. For any
$\varphi\in H^1(\Omega_\varepsilon)$ such that $\varphi=0$ on
$\partial\Omega_\varepsilon\setminus\partial\Omega_0$, i.e., $\varphi=0$ on the rough boundary
$\Gamma_\varepsilon$, we extend it to $(0,1)\times(-\varepsilon,1)$ by zero across the rough boundary
$\Gamma_\varepsilon$, then we can use the Sobolev-Poinc\'{a}re's inequality for $\varphi$ in the flat domains $(0,1)\times(-\varepsilon,1)$ and $(0,1)\times(-\varepsilon,0)$ due to $\varphi=0$ on $(0,1)\times\{-\varepsilon\}$.
\section{First Order Velocity  Corrections}

Recall that $\Gamma_0=:(0,1)\times \{0\}$, $\Gamma_1=:(0,1)\times \{1\}$,  $\Gamma_\varepsilon=:\partial \Omega_\varepsilon\setminus\partial\Omega_0$, $\Sigma_0=:\{0\}\times (0,1)$ and $\Sigma_1=:\{1\}\times (0,1)$.

Before the existence result of the Equation \eqref{1.1}, we first introduce the following auxiliary lemma, whose proof can be founded in \cite[Lemma 4]{MR1813101}.

\begin{lemma}\label{l2.1}
Let $\varphi\in H^1(\Omega_\varepsilon\setminus\Omega_0)$ satisfying $\varphi=0$ on $\partial\Omega_\varepsilon\setminus\partial\Omega_0$, then we have
\begin{equation}\label{2.1}
||\varphi||_{L^2(\Omega_\varepsilon\setminus\Omega_0)}\leq C\varepsilon||\nabla\varphi||_{L^2(\Omega_\varepsilon\setminus\Omega_0)},
\end{equation}
\begin{equation}\label{2.2}
||\varphi||_{L^2(\Gamma_0)}\leq C\varepsilon^{1/2}||\nabla\varphi||_{L^2(\Omega_\varepsilon\setminus\Omega_0)},
\end{equation}
\begin{equation}\label{2.3}
\int_0^1|\varphi(x_1,0)|dx_1\leq C\varepsilon^{1/2}||\partial_2\varphi||_{L^2(\Omega_\varepsilon\setminus\Omega_0)}.
\end{equation}
\end{lemma}

Now, we are ready to prove the desired non-linear stability result of Equation \eqref{1.1}:
\begin{thm}\label{t2.2}
There exists a universal small constant $C_1'$ such that for $|p_1-p_0|\leq C'_1$ and $\varepsilon\leq \varepsilon_0$ with $\varepsilon_0$ being a universal suitably small constant, the problem \eqref{1.1} has a unique solution $\left(U_{S,\varepsilon},P_{S,\varepsilon}\right)\in H^1(\Omega_\varepsilon)^2\times L^2(\Omega_\varepsilon)$ satisfying
\begin{equation}
\begin{aligned}\label{2.4}
\int_{\Omega_\varepsilon}|\nabla U_{S,\varepsilon}-\nabla U_{S,0}|^2\leq C\varepsilon.
\end{aligned}\end{equation}

Moreover, by duality, we also have
\begin{equation}\label{2.5}
\begin{aligned}
\int_{\Omega_\varepsilon}| U_{S,\varepsilon}- U_{S,0}|^2\leq C\varepsilon^2.
\end{aligned}\end{equation}
\end{thm}

\begin{proof}
Note that we do not need the correctors to eliminate these boundary-layer terms in the proof of Theorem \ref{2.2}, then following the similar ideas in \cite[Propositions 1-2]{MR1813101}, we would obtain these estimates, which we omit for simplicity.

\end{proof}

To continue, we introduce the following first-order boundary layer:
\begin{equation}\label{2.6}\left\{\begin{aligned}
-\Delta V+\nabla \Pi&=0\quad \text{ in }\Omega,\\
\operatorname{div}V&=0\quad \text{ in }\Omega,\\
V+(y_2,0)&=0\quad \text{ on }\partial\Omega.
\end{aligned}\right.\end{equation}

We collect some useful properties for the first-order corrector $V$ as below, whose proof can be founded in \cite[Theorems 4.1-4.2]{higaki2021large} and \cite[Proposition 3.1]{MR4619004}.
\begin{prop}\label{p2.3}
Let $(L,K)\in (0,\infty)^2$ and $\Omega\subset \mathbb{R}^2$ be a bumpy John domain with constant $(L,K)$ according to Definition \ref{d1.7}. Then there exists a unique weak solution  $(V,\Pi)\in H^1_{\text{loc}}(\bar{\Omega})^2\times L^2_{\text{loc}}(\bar{\Omega})$ of \eqref{2.6} satisfying
\begin{equation}\label{2.7}
\sup_{\xi\in\mathbb{Z}}\int_{\Omega\cap (\xi+(0,1))\times \mathbb{R}}\left(|\nabla V|^2+|\Pi|^2\right)\leq C,
\end{equation}
where the constant $C$ depends only on $(L,K)$.

 Moreover, if the domain $\Omega$ is periodic as in Definition \ref{d1.8}, then

(i) $(V,\Pi)$ is periodic in $y_1$ and there exists a constant vector $\alpha=:(\alpha_1,0)$, such that
\begin{equation}\label{2.8}\begin{aligned}
|V(y)-\alpha|+|\nabla V(y)|+|\Pi(y)|\leq C||V(\cdot,0)||_{L^2(-1,1)}e^{-y_2/2}\quad \text{for }y_2>1;
\end{aligned}\end{equation}

(ii) for any $1\leq p\leq 2$, we have
\begin{equation}\label{2.8-1}
\sup_{\xi\in\mathbb{Z}}\int_{\Omega\cap (\xi+(0,1))\times \mathbb{R}}\left(|V-\alpha|^p+|\nabla V|^p+|\Pi|^p\right)\leq C.
\end{equation}
Here $C$ depends only on $(L,K)$.
\end{prop}
\begin{proof}The first two estimates \eqref{2.7}-\eqref{2.8} can be found in \cite[Theorems 4.1-4.2]{higaki2021large} and \cite[Proposition 3.1]{MR4619004}. Now we need only to prove the estimate \eqref{2.8-1}. Due to the periodicity of $V$ and $\Pi$ in $y_1$, the estimate \eqref{2.8-1} is equivalent to
\begin{equation}\label{2.9*}
\int_{\Omega\cap (0,1)\times \mathbb{R}}\left(|V-\alpha|^p+|\nabla V|^p+|\Pi|^p\right)\leq C.\tag{2.9*}
\end{equation}

Noting that $\Omega\subset (-\infty,+\infty)\times (-1,+\infty)$, a direct computation shows that
$$\begin{aligned}
&\int_{\Omega\cap (0,1)\times \mathbb{R}}\left(|V-\alpha|^p+|\nabla V|^p+|\Pi|^p\right)\\[5pt]
=&\int_{\Omega\cap (0,1)\times (-1,1)}\left(|V-\alpha|^p+|\nabla V|^p+|\Pi|^p\right)+\int_{\Omega\cap (0,1)\times (1,+\infty)}\left(|V-\alpha|^p+|\nabla V|^p+|\Pi|^p\right)\\[5pt]
\leq&\int_{\Omega\cap (0,1)\times (-1,1)}|V-\alpha|^p+C,\\[5pt]
\end{aligned}$$
where we have used \eqref{2.7}, \eqref{2.8} and Hold\"{e}r's inequality in the inequality above.

To proceed, due to the boundary condition \eqref{2.6} satisfied by $V$, there holds
$$\begin{aligned}
&\int_{\Omega\cap (0,1)\times (-1,1)}|V-\alpha|^p\\[5pt]
\leq&C\int_{\Omega\cap (0,1)\times (-1,1)}|V-(y_2,0)|^p+C\int_{\Omega\cap (0,1)\times (-1,1)}\left(|y_2|^p+|\alpha_1|^p\right)\\[5pt]
\leq&C\int_{\Omega\cap (0,1)\times (-1,1)}|\nabla V-\nabla(y_2,0)|^p+C\leq C,
\end{aligned}$$
where we have used Poincar\'{e}'s inequality and \eqref{2.7} in the inequality above.
Consequently, combining the estimates above yields the desired estimate \eqref{2.9*}.

\end{proof}

To proceed, we introduce the following corrector $\chi_c(x)$ satisfying
\begin{equation}\label{2.9}
\left\{\begin{aligned}
-\Delta \chi_c +\nabla g_c&=0\quad  \text{ in }\Omega_0,\\
\operatorname{div}\chi_c&=0\quad  \text{ in }\Omega_0,\\
\chi_{c,2}&=0\quad  \text{ on }\partial \Omega_0,\\
\chi_{c,1}=0 \text{ on }\Gamma_1& \text{ and }\chi_{c,1}=1 \text{ on }\Gamma_0,\\
g_c=0\text{ on }\Sigma_0\ &\text{ and }\ g_c=0 \text{ on } \Sigma_1.
\end{aligned}\right.\end{equation}

It is easy to see that the Equation \eqref{2.9} admits a unique solution $(\chi_c,g_c)=(1-x_2,0,0)$. Moreover,  we let $\chi_c=(1,0)$ for $x\in \Omega_\varepsilon\setminus\Omega_0$ and denote

\begin{equation}\label{2.10}\begin{aligned}\tilde{W}_{S,\varepsilon}=:U_{S,\varepsilon}&-U_{S,0}-\varepsilon\left(V(x/\varepsilon)-\alpha\right)
\partial_{x_2} U_{S,0,1}(0)\\
&-x_2 I_{x_2\leq 0}\partial_{x_2} U_{S,0,1}(0)e_1-\varepsilon\alpha_1\chi_c(x)\partial_{x_2} U_{S,0,1}(0),\end{aligned}\end{equation}
where $I_{x_2\leq 0}$ is the characteristic function.

Note that $\tilde{W}_{S,\varepsilon,2}\neq 0$ on the  inflow/outflow boundaries $(\{0\}\cup\{1\})\times(0,1)$, now we are going to correct the values of $\tilde{W}_{S,\varepsilon,2}$ there. For this purpose, we introduce the inner boundary layer in $(0,\ell)\times(0,\infty)$, $\ell\ll 1$ (note that $\ell$ depends on the geometry of $\partial \Omega$),
 \begin{equation}\label{2.11}\begin{aligned}
S_1^{\text{in}}(y)&=-\frac\ell 3\left(1-\frac{y_1}\ell\right)^3\frac{\partial V_1}{\partial y_1}(0,y_2),\\[5pt]
S_2^{\text{in}}(y)&=\left(1-\frac{y_1}\ell\right)^2V_2(0,y_2);\\[5pt]
\end{aligned}\end{equation}
and the outter
boundary layer in $(1-\ell,1)\times(0,\infty)$, $\ell\ll 1,$

\begin{equation}\label{2.12}\begin{aligned}
S_1^{\text{out}}(y)&=\frac{\ell} 3\left(1-\frac{1-y_1}{\ell}\right)^3\frac{\partial V_1}{\partial y_1}(1/\varepsilon,y_2),\\[5pt]
S_2^{\text{out}}(y)&=\left(1-\frac{1-y_1}\ell\right)^2V_2(1/\varepsilon,y_2);\\[5pt]
\end{aligned}\end{equation}
Obviously,
$$\operatorname{div}S^{\text{in}}=0\text{ in } (0,\ell)\times(0,\infty);\quad \operatorname{div}S^{\text{out}}=0\text{ in } (1-\ell,1)\times(0,\infty).$$

In view of the boundary condition $\eqref{2.6}_3$ (i.e., the third line of \eqref{2.6}) satisfied by the corrector $V$ and the condition (ii) in Definition \ref{d1.9},
we can make an incompressible $H^1$-extension of $S^{\text{in}}$ to a function defined on $\mathbb{Y}\cap (0,\ell)\times(-1,0)$ and having the zero trace on $\partial [\mathbb{Y}\cap (0,\ell)\times(-1,0)]$, where the notation $\mathbb{Y}$ denotes the unit periodic cell of $\Omega\cap \{(0,1)\times (-1,0)\}$.

Then, we set
\begin{equation}\label{2.13}
S^{\text{in},\varepsilon}(x)=\varepsilon S^{\text{in}}(x/\varepsilon),\ x_1\in[0,\varepsilon\ell);\quad S^{\text{in},\varepsilon}(x)=0,\ x_1\in[\varepsilon\ell,1-\varepsilon\ell],
\end{equation}
and for $x_1\in(1-\varepsilon\ell,1]$,
\begin{equation}\label{2.14}
\begin{aligned}
S_1^{\text{out},\varepsilon}(x)&=\frac{\varepsilon\ell } 3\left(1-\frac{1-x_1}{\varepsilon\ell}\right)^3\frac{\partial V_1}{\partial y_1}(1/\varepsilon,x_2/\varepsilon),\\[5pt]
S_2^{\text{out},\varepsilon}(x)&=\varepsilon\left(1-\frac{1-x_1}{\varepsilon\ell}\right)^2V_2(1/\varepsilon,x_2/\varepsilon);\\[5pt]
\end{aligned}\end{equation}
and
\begin{equation}\label{2.15}
S^{\text{out},\varepsilon}(x)=0,\ x_1\in[0,1-\varepsilon\ell].
\end{equation}
Then, for every $q\in[1,\infty)$, we have
\begin{equation}\label{2.16}
\begin{aligned}
\varepsilon^{-1}||S^{\varepsilon}||_{L^q(\Omega_\varepsilon)}
+||\nabla S^{\varepsilon}||_{L^q(\Omega_\varepsilon)}\leq C\varepsilon^{2/q},\\[5pt]
\end{aligned}
\end{equation}
with $S^{\varepsilon}=S^{\text{in},\varepsilon}\text{ or }S^{\text{in},\varepsilon}$. Now we set
\begin{equation}\label{2.17}\begin{aligned}
\mathcal{W}_{S,\varepsilon}=:\tilde{W}_{S,\varepsilon}+S^{\text{in},\varepsilon}\partial_{x_2} U_{S,0,1}(0)
+S^{\text{out},\varepsilon}\partial_{x_2} U_{S,0,1}(0).
\end{aligned}\end{equation}
Note that, by definition, we have $\mathcal{W}_{S,\varepsilon}=0$ on the oscillating boundary
$\Gamma_\varepsilon$ and $||\mathcal{W}_{S,\varepsilon}||_{L^\infty}=O(\exp\{-C\varepsilon^{-1}\})$ on the
flat boundary $\Gamma_1$. Moreover, we know that  $\operatorname{div}\mathcal{W}_{S,\varepsilon}=0$ in
$\Omega_\varepsilon$, and $\mathcal{W}_{S,\varepsilon,2}= 0$ on the  inflow/outflow boundaries
$(\{0\}\cup\{1\})\times(0,1)$. Now, $\mathcal{W}_{S,\varepsilon}$ is a suitable test function in \eqref{2.19}.

\begin{thm}\label{t2.4}
Under the conditions in Theorem \ref{t1.1} and
with the notations above, there holds
\begin{equation}\label{2.18}
\begin{aligned}
\int_{\Omega_\varepsilon}|\nabla \mathcal{W}_{S,\varepsilon}|^2
\leq& C\varepsilon^{3}.
\end{aligned}\end{equation}
\end{thm}

\begin{proof}
First, a direct computation yields that  $W_{S,\varepsilon}=:U_{S,\varepsilon}-U_{S,0}$ satisfies the following variational equation:

\begin{equation}\label{2.19}
\begin{aligned}
\int_{\Omega_\varepsilon}\nabla W_{S,\varepsilon}& \nabla \varphi
+\int_{\Omega_\varepsilon}\left(W_{S,\varepsilon}\cdot \nabla  W_{S,\varepsilon}+W_{S,\varepsilon,2}\partial_2  U_{S,0,1}e_1+U_{S,0,1}\partial_1 W_{S,\varepsilon} \right)\cdot \varphi\\
&=\partial_2 U_{S,0,1}(0)\int_{\Gamma_0}\varphi_1-(p_1-p_0)\int_{\Omega_\varepsilon\setminus \Omega_0}  \varphi_1,\quad \forall \varphi\in \mathcal{H}_\varepsilon,
\end{aligned}\end{equation}
with $\mathcal{H}_\varepsilon=:\left\{u\in H^1(\Omega_\varepsilon)^2,
u=0 \text{ on } \{\Gamma_1\cup \Gamma_\varepsilon\},\ u_2=0\text{ on } \{\Sigma_0\cup \Sigma_1\},\ \operatorname{div}u=0 \text{ in } \Omega_\varepsilon \right\}$.

To continue, let $n$ denote the outward unit normal to $\partial \Omega_\varepsilon$, then a careful computation yields that
\begin{equation}\label{2.20}
\begin{aligned}
&-\int_{\Omega_\varepsilon}\nabla_y V^\varepsilon \nabla \mathcal{W}_{S,\varepsilon}\cdot\partial_2 {U}_{S,0,1}(0)\\
=&-\int_{\Gamma_1}\frac{\partial (\varepsilon V^\varepsilon)}{\partial n}\mathcal{W}_{S,\varepsilon}\cdot\partial_2 {U}_{S,0,1}(0)-\int_{\Sigma_1}\partial_1 (\varepsilon V^\varepsilon)\cdot\mathcal{W}_{S,\varepsilon}\cdot\partial_2 {U}_{S,0,1}(0)\\
&\quad+\int_{\Sigma_0}\partial_1 (\varepsilon V^\varepsilon)\cdot\mathcal{W}_{S,\varepsilon}\cdot\partial_2 {U}_{S,0,1}(0)\\
=&:I_1+I_{2,1}+I_{2,2}+I_{3,1}+I_{3,2},
\end{aligned}\end{equation}
with
$$I_{2,1}=:-\int_{\Sigma_1}\partial_1 (\varepsilon V_1^\varepsilon)\cdot\mathcal{W}_{S,\varepsilon,1}\cdot\partial_2 {U}_{S,0,1}(0),$$
and analogously for other $I_{i,j}$ for $i=2,3$ and $j=1,2$.

It is easy to see that
\begin{equation}\label{2.21}
\left|I_1\right|\leq C\exp\{-C\varepsilon^{-1}\}.
\end{equation}
 Using the boundary conditions $\mathcal{W}_{S,\varepsilon,2}=0$ on $\{\Sigma_0\cup \Sigma_1\}$ yields that
\begin{equation}\label{2.22}
\begin{aligned}
I_{2,2}+I_{3,2}&=-\int_{\Sigma_1}\partial_1 (\varepsilon V_2^\varepsilon)\cdot\mathcal{W}_{\varepsilon,2}\cdot\partial_2 {U}_{S,0,1}(0)
+\int_{\Sigma_0}\partial_1 (\varepsilon V_2^\varepsilon)\cdot\mathcal{W}_{\varepsilon,2}\cdot\partial_2 {U}_{S,0,1}(0)\\
&=0.
\end{aligned}\end{equation}

To proceed, using the similar notations as $I_{i,j}$ for $i=2,3$ and $j=1,2$, a direct computation shows that
\begin{equation}\label{2.23}
\begin{aligned}
&-\int_{\Omega_\varepsilon}\nabla_x \left(S^{\text{in},\varepsilon}+S^{\text{out},\varepsilon}\right) \nabla \mathcal{W}_{S,\varepsilon}\cdot\partial_2 {U}_{S,0,1}(0)\\
=&-\int_{\Gamma_1}\frac{\partial\left(S^{\text{in},\varepsilon}+S^{\text{out},\varepsilon}\right)}{\partial n}\mathcal{W}_{S,\varepsilon}\cdot\partial_2 {U}_{S,0,1}(0)+\int_{\Sigma_0}\partial_1 S^{\text{in},\varepsilon}\cdot\mathcal{W}_{S,\varepsilon}\cdot\partial_2 {U}_{S,0,1}(0)\\
&\quad-\int_{\Sigma_1}\partial_1 S^{\text{out},\varepsilon}\cdot\mathcal{W}_{S,\varepsilon}\cdot\partial_2 {U}_{S,0,1}(0)\\
=&:I_4+I_{5,1}+I_{5,2}+I_{6,1}+I_{6,2}.
\end{aligned}\end{equation}
Similar to the explanation of \eqref{2.21}, we have
\begin{equation}\label{2.24}
\left| I_4\right|\leq C\exp\{-C\varepsilon^{-1}\}.
\end{equation}
In view of the definition of $S^{\text{in},\varepsilon}$ in \eqref{2.13}, we have
\begin{equation}\label{2.25}\begin{aligned}
I_{5,1}=\int_{\Sigma_0}\partial_1 S^{\text{in},\varepsilon}_1\cdot\mathcal{W}_{S,\varepsilon,1}\cdot\partial_2 {U}_{S,0,1}(0)=\int_{\Sigma_0}\frac{\partial V_1}{\partial y_1}(0,x_2/\varepsilon)\cdot\mathcal{W}_{S,\varepsilon,1}\cdot\partial_2 {U}_{S,0,1}(0)=I_{3,1},
\end{aligned}\end{equation}
and in view of the definition of $S^{\text{out},\varepsilon}$ in \eqref{2.14}, we have
\begin{equation}\label{2.26}\begin{aligned}
I_{6,1}=-\int_{\Sigma_1}\partial_1
S^{\text{out},\varepsilon}_1\cdot\mathcal{W}_{S,\varepsilon,1}\cdot\partial_2
{U}_{S,0,1}(0)=-\int_{\Sigma_1}\frac{\partial V_1}{\partial
y_1}\left(\frac 1 \varepsilon,\frac{x_2}\varepsilon\right)\cdot\mathcal{W}_{S,\varepsilon,1}\cdot\partial_2 {U}_{S,0,1}(0)=I_{2,1}.
\end{aligned}\end{equation}
Note that
$\mathcal{W}_{S,\varepsilon,2}=0$ on $\{\Sigma_0\cup \Sigma_1\}$, then we have
\begin{equation}\label{2.27}
I_{5,2}+I_{6,2}=0.
\end{equation}
Therefore, combining \eqref{2.20}-\eqref{2.27} yields that
\begin{equation}\label{2.28}
\left|-\int_{\Omega_\varepsilon}\nabla_x \left(\varepsilon V^\varepsilon-S^{\text{in},\varepsilon}-S^{\text{out},\varepsilon}\right) \nabla \mathcal{W}_{S,\varepsilon}\cdot\partial_2 {U}_{S,0,1}(0)\right|\leq \exp\{-C\varepsilon^{-1}\}.
\end{equation}

%
%

\noindent
To proceed, a direct computation shows that
\begin{equation}\label{2.29}\begin{aligned}
-\int_{\Omega_\varepsilon}\nabla(x_2I_{x_2\leq 0}e_1)\nabla \mathcal{W}_{S,\varepsilon} \partial_{x_2} {U}_{S,0,1}(0)
&=-\int_{\Omega_\varepsilon\setminus \Omega_0}\partial_2 \mathcal{W}_{S,\varepsilon,1}\partial_{x_2} {U}_{S,0,1}(0)\\
&=-\int_{\Gamma_0}\mathcal{W}_{S,\varepsilon,1}\partial_{x_2} {U}_{S,0,1}(0);
\end{aligned}\end{equation}
and

\begin{equation}\label{2.30}\begin{aligned}
-\int_{\Omega_\varepsilon}\varepsilon\alpha_1\partial_{x_2} {U}_{S,0,1}(0)\nabla \chi_c\nabla \mathcal{W}_{S,\varepsilon}
=&-\int_{\Omega_0}\varepsilon\alpha_1\partial_{x_2} {U}_{S,0,1}(0)\nabla (1-x_2)e_1\cdot\nabla \mathcal{W}_{S,\varepsilon}\\
=&\int_{\Omega_0}\varepsilon\alpha_1\partial_{x_2} {U}_{S,0,1}(0)\partial_2 \mathcal{W}_{S,\varepsilon,1}\\
=&\int_{\Gamma_1}\varepsilon\alpha_1\partial_{x_2} {U}_{S,0,1}(0)\mathcal{W}_{S,\varepsilon,1}-\int_{\Gamma_0}\varepsilon\alpha_1\partial_{x_2} {U}_{S,0,1}(0)\mathcal{W}_{S,\varepsilon,1}\\
=&:J_1+J_2.
\end{aligned}\end{equation}
It is easy to see that

\begin{equation}\label{2.31}
\left| J_1\right|\leq O(\exp\{-C\varepsilon^{-1}\}).
\end{equation}
Due to \eqref{2.2}, we know
\begin{equation}\label{2.32}
\left| J_2\right|\leq C\varepsilon ^{3/2} ||\nabla\mathcal{W}_{S,\varepsilon}||_{L^2(\Omega_\varepsilon)}.
\end{equation}

Now, choosing $\varphi=\mathcal{W}_{S,\varepsilon}$ in \eqref{2.19} after noting that $\mathcal{W}_{S,\varepsilon}\neq 0$  for $x\in {\Gamma_1}$ gives that

\begin{equation}\label{2.33}\begin{aligned}
&\int_{\Omega_\varepsilon}\nabla W_{S,\varepsilon} \nabla \mathcal{W}_{S,\varepsilon}
+\int_{\Omega_\varepsilon}\left(W_{S,\varepsilon}\cdot \nabla  W_{S,\varepsilon}+W_{S,\varepsilon,2}\partial_2  U_{S,0,1}e_1+U_{S,0,1}\partial_1 W_{S,\varepsilon} \right)\cdot \mathcal{W}_{S,\varepsilon}\\
=&\partial_2 U_{S,0,1}(0)\int_{\Gamma_0}\mathcal{W}_{S,\varepsilon,1}-(p_1-p_0)\int_{\Omega_\varepsilon\setminus \Omega_0}  \mathcal{W}_{S,\varepsilon,1}+\int_{\Gamma_1}\left(\frac{\partial {W}_{\varepsilon}}{\partial n}-P_{S,\varepsilon}I_{2\times 2}n\right)\mathcal{W}_{S,\varepsilon},
\end{aligned}\end{equation}
with $I_{2\times 2}=:
\left(\begin{matrix}

      1 & 0 \\

       0 & 1

  \end{matrix}\right)$.

Using Poincar\'{e}'s inequality  yields that
\begin{equation}\label{2.34}\begin{aligned}
\left|(p_1-p_0)\int_{\Omega_\varepsilon\setminus \Omega_0}  \mathcal{W}_{S,\varepsilon,1}\right|
\leq \varepsilon|p_1-p_0|\int_{\Omega_\varepsilon\setminus \Omega_0}|\nabla \mathcal{W}_{S,\varepsilon,1}|
\leq C\varepsilon^{3/2}||\nabla\mathcal{W}_{S,\varepsilon}||_{L^2(\Omega_\varepsilon)}.
\end{aligned}\end{equation}

It is easy to see that on the flat boundary $\Gamma_1$, we have
\begin{equation}\label{2.35}
\left|\int_{\Gamma_1}\left(\frac{\partial {W}_{\varepsilon}}{\partial n}-P_{S,\varepsilon}I_{2\times 2}n\right)\mathcal{W}_{S,\varepsilon}\right|\leq O(\exp\{-C\varepsilon^{-1}\}).
\end{equation}

Now we need to estimate the terms in the first line of \eqref{2.33}. First, a direct computation shows that
\begin{equation}\label{2.36}\begin{aligned}
&\left|\int_{\Omega_\varepsilon}W_{S,\varepsilon}\cdot \nabla  W_{S,\varepsilon}\cdot \mathcal{W}_{S,\varepsilon}\right|\\
\leq& \left|\int_{\Omega_\varepsilon}\left(W_{S,\varepsilon}-\mathcal{W}_{S,\varepsilon}\right)\cdot \nabla  W_{S,\varepsilon}\cdot \mathcal{W}_{S,\varepsilon}\right|+\left|\int_{\Omega_\varepsilon}\mathcal{W}_{S,\varepsilon}\cdot \nabla  W_{S,\varepsilon}\cdot \mathcal{W}_{S,\varepsilon}\right|\\[5pt]
\leq& ||W_{S,\varepsilon}-\mathcal{W}_{S,\varepsilon}||_{L^4(\Omega_\varepsilon)}||\nabla
W_{S,\varepsilon}||_{L^2(\Omega_\varepsilon)}
||\mathcal{W}_{S,\varepsilon}||_{L^4(\Omega_\varepsilon)}\\[5pt]
&\quad\quad\quad\quad\quad\quad\quad+||\mathcal{W}_{S,\varepsilon}||_{L^4(\Omega_\varepsilon)}||\nabla
W_{S,\varepsilon}||_{L^2(\Omega_\varepsilon)}
||\mathcal{W}_{S,\varepsilon}||_{L^4(\Omega_\varepsilon)}\\[5pt]
\leq& C\varepsilon\cdot
\varepsilon^{1/2}||\nabla\mathcal{W}_{S,\varepsilon}||_{L^2(\Omega_\varepsilon)}
+C\varepsilon^{1/2}||\nabla\mathcal{W}_{S,\varepsilon}||^2_{L^2(\Omega_\varepsilon)},\\[5pt]
\end{aligned}\end{equation}
where we have used \eqref{2.4} and
$$\begin{aligned}\left\|W_{S,\varepsilon}-\mathcal{W}_{S,\varepsilon}\right\|_{L^4(\Omega_\varepsilon)}
=&\left\|\left[\varepsilon\left(V(x/\varepsilon)-\alpha\right)+
x_2 I_{x_2\leq 0}e_1+\varepsilon\alpha_1\chi_c(x)\right.\right.\\
&\quad\quad\quad\quad\quad\quad\quad\quad\quad\quad\quad\quad\quad\left.\left.-S^{\text{in},\varepsilon}
-S^{\text{out},\varepsilon}\right] \partial_{x_2}U_{S,0,1}(0)\right\|_{L^4(\Omega_\varepsilon)}\\
\leq& C\varepsilon\end{aligned}$$
in the inequality above.

Next, we have
\begin{equation}\label{2.37}\begin{aligned}
&\left|\int_{\Omega_\varepsilon}W_{S,\varepsilon,2}\partial_2  U_{S,0,1}e_1 \cdot \mathcal{W}_{S,\varepsilon}\right|=\left|\int_{\Omega_\varepsilon}W_{S,\varepsilon,2}\partial_2  U_{S,0,1}\mathcal{W}_{S,\varepsilon,1}\right|\\[5pt]
\leq&\left|\int_{\Omega_\varepsilon}\mathcal{W}_{S,\varepsilon,2}\partial_2  U_{S,0,1}\mathcal{W}_{S,\varepsilon,1}\right|
+\left|\int_{\Omega_\varepsilon}\left(W_{S,\varepsilon,2}-\mathcal{W}_{S,\varepsilon,2}\right)\partial_2  U_{S,0,1}\mathcal{W}_{S,\varepsilon,1}\right|\\[5pt]
\leq &C|p_1-p_0|\cdot||\nabla\mathcal{W}_{S,\varepsilon}||_{L^2(\Omega_\varepsilon)}^2+C\varepsilon^{2-} ||\nabla\mathcal{W}_{S,\varepsilon}||_{L^2(\Omega_\varepsilon)},
\end{aligned}\end{equation}
where we have used for any $1< p\leq 2$, \begin{equation}\label{2.38-1}\begin{aligned}||W_{S,\varepsilon,2}-\mathcal{W}_{S,\varepsilon,2}||_{L^p(\Omega_\varepsilon)}=&|\partial_{x_2} U_{S,0,1}(0)|\cdot||\varepsilon V_2(x/\varepsilon)
-S^{\text{in},\varepsilon}_1-S^{\text{out},\varepsilon}_1||_{L^p(\Omega_\varepsilon)}\\[5pt]
\leq& C\varepsilon^{1+1/p}=C\varepsilon^{2-},\\[5pt]
\end{aligned}\end{equation}
in the inequality above. Note that the inequality \eqref{2.38-1} follows from \eqref{2.8-1} and \eqref{2.16}.

To see the estimate on the term $\int_{\Omega_\varepsilon}U_{S,0,1}\partial_1 W_{S,\varepsilon} \cdot \mathcal{W}_{S,\varepsilon}$ after in view of the definition of $\mathcal{W}_{S,\varepsilon}$ in \eqref{2.17},  we have
\begin{equation}\label{2.38}\begin{aligned}
\left|\int_{\Omega_\varepsilon}U_{S,0,1}\partial_1 W_{S,\varepsilon} \cdot
\mathcal{W}_{S,\varepsilon}\right|
\leq& \left|\int_{\Omega_\varepsilon}U_{S,0,1}\partial_1\mathcal{ W}_{S,\varepsilon} \cdot
\mathcal{W}_{S,\varepsilon}\right|+\left|\int_{\Omega_\varepsilon}U_{S,0,1}\partial_1 \left(W_{S,\varepsilon} -\mathcal{W}_{S,\varepsilon}\right) \cdot \mathcal{W}_{S,\varepsilon}\right|\\[5pt]
\leq& C|p_1-p_0|\cdot||\nabla\mathcal{W}_{S,\varepsilon}||^2_{L^2(\Omega_\varepsilon)}+I_{7,1}+I_{7,2},\\[5pt]
\end{aligned}\end{equation}
where
\begin{equation}\label{2.39}\begin{aligned}
I_{7,i}=:\left|\int_{\Omega_\varepsilon}U_{S,0,1}\partial_1 \left(W_{S,\varepsilon,i} -\mathcal{W}_{S,\varepsilon,i}\right) \cdot \mathcal{W}_{S,\varepsilon,i}\right|
\end{aligned}\end{equation}
for $i=1,2$.
A direct computation shows that
\begin{equation}\label{2.40}\begin{aligned}
I_{7,2}=&\left|\int_{\Omega_\varepsilon}U_{S,0,1}\left(W_{S,\varepsilon,2} -\mathcal{W}_{S,\varepsilon,2}\right) \cdot \partial_1 \mathcal{W}_{S,\varepsilon,2}\right|\\[5pt]
\leq &C||W_{S,\varepsilon,2}-\mathcal{W}_{S,\varepsilon,2}||_{L^2(\Omega_\varepsilon)}||\nabla\mathcal{W}_{S,\varepsilon}||_{L^2(\Omega_\varepsilon)}\\[5pt]
\leq & C\varepsilon^{3/2}||\nabla\mathcal{W}_{S,\varepsilon}||_{L^2(\Omega_\varepsilon)},\\[5pt]
\end{aligned}\end{equation}
where we have used \eqref{2.38-1} in the inequality above.

Due to the divergence-free condition and \eqref{2.38-1} again, we have
\begin{equation}\label{2.41}\begin{aligned}
I_{7,1}=&\left|\int_{\Omega_\varepsilon}U_{S,0,1}\partial_1 \left(W_{S,\varepsilon,1} -\mathcal{W}_{S,\varepsilon,1}\right) \cdot \mathcal{W}_{S,\varepsilon,1}\right|\\[5pt]
=&\left|\int_{\Omega_\varepsilon}U_{S,0,1}\partial_2 \left(W_{S,\varepsilon,2}
-\mathcal{W}_{S,\varepsilon,2}\right) \cdot \mathcal{W}_{S,\varepsilon,1}\right|\\[5pt]
\leq &\left|\int_{\Gamma_1}U_{S,0,1} \left(W_{S,\varepsilon,2} -\mathcal{W}_{S,\varepsilon,2}\right)
\mathcal{W}_{S,\varepsilon,1}\right| +\left|\int_{\Omega_\varepsilon}\partial_2 U_{S,0,1}
\left(W_{S,\varepsilon,2} -\mathcal{W}_{S,\varepsilon,2}\right) \cdot \mathcal{W}_{S,\varepsilon,1}\right|\\
&\quad\quad\quad\quad \quad\quad\quad\quad+\left|\int_{\Omega_\varepsilon}U_{S,0,1} \left(W_{S,\varepsilon,2} -\mathcal{W}_{S,\varepsilon,2}\right) \cdot \partial_2\mathcal{W}_{S,\varepsilon,1}\right|\\[5pt]
\leq &O(\exp\{-C\varepsilon^{-1}\})+C||W_{S,\varepsilon,2} -\mathcal{W}_{S,\varepsilon,2}||_{L^2(\Omega_\varepsilon)}
||\mathcal{W}_{S,\varepsilon,1}||_{L^{2}(\Omega_\varepsilon)}\\[5pt]
&\quad\quad+C||W_{S,\varepsilon,2} -\mathcal{W}_{S,\varepsilon,2}||_{L^2(\Omega_\varepsilon)}
||\partial_2\mathcal{W}_{S,\varepsilon,1}||_{L^{2}(\Omega_\varepsilon)}\\[5pt]
\leq & C\varepsilon^{3/2}||\nabla\mathcal{W}_{S,\varepsilon}||_{L^2(\Omega_\varepsilon)}+O(\exp\{-C\varepsilon^{-1}\}).\\[5pt]
\end{aligned}\end{equation}

Consequently,  combining \eqref{2.19}-\eqref{2.41} after choosing $|p_1-p_0|+\varepsilon^{1/2}$ suitably small yields that
\begin{equation*}
\int_{\Omega_\varepsilon}|\nabla \mathcal{W}_{S,\varepsilon}|^2
\leq C\varepsilon^{3},
\end{equation*} which completes the proof of Theorem \ref{t2.4}.
\end{proof}


Now we are ready to prove Theorem \ref{t1.1} and Corollary \ref{c1.2}.

\noindent \textbf{Proof of Theorem \ref{t1.1} and Corollary \ref{c1.2}}: Due to  \eqref{2.2} and \eqref{2.18}, we have
\begin{equation*}
|| \mathcal{W}_{S,\varepsilon}||_{L^2(\Gamma_0)}
\leq C\varepsilon^{2}.
\end{equation*}

To proceed, by a duality argument similar to \cite[Equations (73)-(74)]{MR1813101} and \cite{MR0884812,MR0884813} with a similar approximation of the pressure, there holds
\begin{equation}\label{2.42}
|| \mathcal{W}_{S,\varepsilon}||_{L^2(\Omega_0)}
\leq C\varepsilon^{2-}.
\end{equation}
Note that in the duality argument, the $O(\varepsilon^{2-})$ convergence rates come from
$||\nabla S^{\varepsilon}||_{L^q(\Omega_\varepsilon)}\leq C\varepsilon^{2/q}$
for any $2\geq p>1$ with $S^{\varepsilon}=S^{\text{in},\varepsilon}\text{ or }S^{\text{in},\varepsilon}$.

Consequently, in view of \eqref{2.16} and the definition of $\mathcal{W}_{S,\varepsilon}$ in \eqref{2.17}, we have
\begin{equation}\label{2.43}||U_{S,\varepsilon}-U_{S,0}-\varepsilon\left(V(x/\varepsilon)-\alpha\right)
\partial_{x_2} U_{S,0,1}(0)-\varepsilon\alpha_1\chi_c(x)\partial_{x_2} U_{S,0,1}(0)||_{L^2(\Omega_0)}\leq C \varepsilon^{2-},\end{equation}
which completes the proof of Theorem \ref{t1.1} after noting that $\chi_c=(1-x_2,0,0)$ and $\partial_{x_2} U_{S,0,1}(0)=-\frac12(p_1-p_0)$.

 Next, in view of the effective model $U_{s,\text{eff}}$ defined in \eqref{1.3}, a direct computation shows that

\begin{equation}\label{2.44}\begin{aligned}
&U_{S,\text{eff},1}-U_{S,0,1}-\varepsilon\alpha_1\chi_{c,1}(x)\partial_{x_2} U_{S,0,1}(0)\\
=&\frac{p_1-p_0}{2}\left(x_2^2-(x_2+\varepsilon\alpha_1-x_2(x_2-1)+\varepsilon\alpha_1(1-x_2))
\frac{1}{1+\varepsilon\alpha_1}\right)\\
=&\frac{p_1-p_0}{2}\cdot \varepsilon^2\alpha_1^2 \cdot \frac{1-x_2}{1+\varepsilon\alpha_1}.
\end{aligned}\end{equation}

Moreover,
\begin{equation}\begin{aligned}\label{2.45}
U_{S,\varepsilon}-U_{S,\text{eff}}=&U_{S,\varepsilon}-U_{S,0}-\varepsilon\left(V(x/\varepsilon)-\alpha\right)
\partial_{x_2} U_{S,0,1}(0)-\varepsilon\alpha_1\chi_c(x)\partial_{x_2} U_{S,0,1}(0)\\[5pt]
&\quad\quad+U_{S,0}+\varepsilon\alpha_1\chi_c(x)\partial_{x_2} U_{S,0,1}(0)-U_{S,\text{eff}}\\[5pt]
&\quad\quad\quad\quad\quad\quad+\varepsilon\left(V(x/\varepsilon)-\alpha\right)
\partial_{x_2} U_{S,0,1}(0).\\[5pt]
\end{aligned}\end{equation}

Now, Corollary \ref{c1.2} follows readily from \eqref{2.8-1} and \eqref{2.43}-\eqref{2.45}. Meanwhile, we can also deduce the same Navier's consitions as in \cite{MR1813101}, which we do not persue for simplicity.\qed

\section{Proof of Corollary \ref{c1.4}}
First of all, we need to prove the desired non-linear stability result of Equation \eqref{1.4}, whose solution is a
perturbation of the solution to the Equation \eqref{1.1}. However, we do not pay attention to the local existence and
uniqueness of the solution to the Equation \eqref{1.4}, and we only focus on the global existence and the decay estimates. (Actually, by considering the variational equation satisfied by $U_{N,\varepsilon}-U_{S,\varepsilon}$,
similar to the idea in \cite[Proposition 1]{MR1813101}, we would obtain the desired local existence and uniqueness in $\{L^\infty(0,T;L^2(\Omega_\varepsilon))^2\cap L^2(0,T;H^1(\Omega_\varepsilon))^2\}\times L^2(\Omega_\varepsilon\times(0,T))$ for some $T>0$.)

Denote $U_\varepsilon=:U_{N,\varepsilon}-U_{S,\varepsilon}$, $P_\varepsilon=:P_{N,\varepsilon}-P_{S,\varepsilon}$ and $T^*<\infty$ is the lifetime such that
\begin{equation}\label{3.1}
T^*=\sup_{T>0}\left\{T>0:||U_\varepsilon||_{L^\infty(0,T;L^2(\Omega_\varepsilon))}\leq (1-\delta/2)/G_N^2\right\}.
\end{equation}

In view of \eqref{1.1} and \eqref{1.4}, a direct computation shows that $(U_\varepsilon,P_\varepsilon)$ satisfies the following equation:
\begin{equation}\label{3.2}
\left\{\begin{aligned}
\partial_t U_{\varepsilon}-\Delta U_{\varepsilon}+U_{\varepsilon}\cdot \nabla  U_{\varepsilon}+\nabla P_{\varepsilon}&=-U_{S,\varepsilon}\cdot \nabla U_\varepsilon-U_\varepsilon\cdot \nabla U_{S,\varepsilon} \quad\text{in }\Omega_\varepsilon\times (0,T^*),\\
\operatorname{div}U_{\varepsilon}&=0 \quad \text{in }\Omega_\varepsilon\times (0,T^*),\\
U_{\varepsilon}&=0 \quad \text{on }\{\Gamma_\varepsilon\cup \Gamma_1\}\times (0,T^*),\\
\ U_{\varepsilon,2}&=0 \quad \text{on }\{\Sigma_0\cup \Sigma_1\}\times (0,T^*),\\
P_{\varepsilon}=0 \text{ on }\{\Sigma_0\}\times (0,T^*),&\quad\ P_{\varepsilon}=0 \text{ on }\{\Sigma_1\}\times (0,T^*),\\
U_{\varepsilon}(x,0)&=\varphi(x)-U_{S,\varepsilon} \text{ on }\Omega_\varepsilon.
\end{aligned}\right.\end{equation}

For any $0<t<T^*$, multiplying the Equation \eqref{3.2} by $U_\varepsilon$ and integrating the resulting equation over $\Omega_\varepsilon$ after using the boundary conditions, we have
\begin{equation}\label{3.3}\begin{aligned}
\frac12 \frac d{dt}\int_{\Omega_\varepsilon}|U_\varepsilon|^2+\int_{\Omega_\varepsilon}|\nabla U_\varepsilon|^2=&\int_{\Omega_\varepsilon}\left(-U_{\varepsilon}\cdot \nabla  U_{\varepsilon}-U_{S,\varepsilon}\cdot \nabla U_\varepsilon-U_\varepsilon\cdot \nabla U_{S,\varepsilon}\right)\cdot U_\varepsilon\\
=&:M_1+M_2+M_3.
\end{aligned}\end{equation}

Now, we need to estimate $M_i$ with $1\leq i\leq 3$ term by term. First, it follows by H\"{o}lder's inequality, Galiardo-Nirenberg's inequality \eqref{1.5} and the definition of $T^*$ in \eqref{3.1}, we have
\begin{equation}\label{3.4}\begin{aligned}
|M_1|\leq& ||U_\varepsilon||_{L^4(\Omega_\varepsilon)}^2||\nabla U_\varepsilon||_{L^2(\Omega_\varepsilon)}
\leq G_N^2||U_\varepsilon||_{L^2(\Omega_\varepsilon)}||\nabla U_\varepsilon||_{L^2(\Omega_\varepsilon)}^2\\[5pt]
\leq & \left(1-\delta/2\right)||\nabla U_\varepsilon||_{L^2(\Omega_\varepsilon)}^2.\\[5pt]
\end{aligned}\end{equation}

 Similarly, it follows by H\"{o}lder's inequality and Poincar\'{e}' inequality, we have
 \begin{equation}\label{3.5}\begin{aligned}
|M_2+M_3|\leq& ||U_{S,\varepsilon}||_{L^4(\Omega_\varepsilon)}||\nabla U_\varepsilon||_{L^2(\Omega_\varepsilon)}|| U_\varepsilon||_{L^4(\Omega_\varepsilon)}+||U_\varepsilon||_{L^4(\Omega_\varepsilon)}^2||\nabla U_{S,\varepsilon}||_{L^2(\Omega_\varepsilon)}\\[5pt]
\leq& C||\nabla U_{S,\varepsilon}||_{L^2(\Omega_\varepsilon)} ||\nabla U_\varepsilon||_{L^2(\Omega_\varepsilon)}^2\\[5pt]
\leq& C\left(||\nabla U_{S,0}||_{L^2(\Omega_\varepsilon)}+||\nabla (U_{S,\varepsilon}-U_{S,0})||_{L^2(\Omega_\varepsilon)}\right)||\nabla U_{\varepsilon}||_{L^2(\Omega_\varepsilon)}^2\\[5pt]
\leq & C_1\left(|p_1-p_0|+\varepsilon^{1/2}\right)||\nabla U_\varepsilon||_{L^2(\Omega_\varepsilon)}^2,\\[5pt]
\end{aligned}\end{equation}
where we have used \eqref{1.2} and \eqref{2.4} in the inequality above.

By choosing $0<\varepsilon\leq \varepsilon_{\delta_1}$ with $\varepsilon_{\delta_1}=\delta^2/(64C_1^2)$ and $|p_1-p_0|<\delta/(8C_1)$ after combining \eqref{3.3}-\eqref{3.5}, we have

\begin{equation}\label{3.6}\begin{aligned}
\frac12 \frac d{dt}\int_{\Omega_\varepsilon}|U_\varepsilon|^2\leq -\frac \delta 4\int_{\Omega_\varepsilon}|\nabla U_\varepsilon|^2\leq -C\delta/2 \int_{\Omega_\varepsilon}|U_\varepsilon|^2.
\end{aligned}\end{equation}

By Gronwall's Lemma, we have
\begin{equation}\label{3.7}
\int_{\Omega_\varepsilon}|U_\varepsilon|^2(\cdot,t)\leq e^{-C\delta t}\int_{\Omega_\varepsilon}|U_\varepsilon|^2(\cdot,0)\leq \frac{(1-7\delta/8)^2}{G_N^4}e^{-C\delta t}.
\end{equation}
where
\begin{equation*}\begin{aligned}
||U_\varepsilon(\cdot,0)||_{L^2(\Omega_\varepsilon)}
\leq& ||\varphi-U_{S,0}||_{L^2(\Omega_\varepsilon)}+||U_{S,\varepsilon}-U_{S,0}||_{L^2(\Omega_\varepsilon)}\\[5pt]
\leq& (1-\delta)/G_N^2+C_2\varepsilon\\[5pt]
\leq &(1-7\delta/8)/G_N^2,\\[5pt]
\end{aligned}\end{equation*}
for any $0<\varepsilon\leq \varepsilon_{\delta_2}=\delta/(8C_2G_N^2)$.

Consequently, be viewing the definition of $T^*$ in \eqref{3.1}, we know that $T^*=\infty$ due to \eqref{3.6}-\eqref{3.7}. Note that the inequality \eqref{3.6} also implies that $U_\varepsilon\in L^\infty(0,\infty;L^2(\Omega_\varepsilon))^2\cap L^2(0,\infty;H^1(\Omega_\varepsilon))^2$. Moreover, for any $0<t<\infty$, we have the following decay estimates:
\begin{equation*}
\int_{\Omega_\varepsilon}|U_\varepsilon|^2(\cdot,t)\leq \frac{(1-7\delta/8)^2}{G_N^4}e^{-C\delta t},
\end{equation*}
which completes the proof of Corollary \ref{c1.4} after noting the results in Corollary \ref{c1.2}.
\begin{center}{\textbf{Acknowledgements}}
\end{center}

The author wants to express his sincere appreciation to  Prof. Jinping Zhuge  for  helpful instructions and discussions.

\normalem\bibliographystyle{plain}{}

\end{document}